\documentclass[noams]{compositio}
\usepackage[all]{xy}
\usepackage{amssymb}
\usepackage{latexsym,amsfonts}

\newcommand{\PGL}{\mathop{\mathrm{PGL}}}
\newcommand{\Var}{\mathop{\mathrm{Var}}}
\newcommand{\Alg}{\mathop{\mathrm{Alg}}}
\newcommand{\Br}{\mathop{\mathrm{Br}}}
\newcommand{\ind}{\mathop{\mathrm{ind}}}
\newcommand{\rdim}{\mathop{\mathrm{rdim}}}
\newcommand{\CH}{\mathop{\mathrm{CH}}}
\newcommand{\Corr}{\mathop{\mathrm{Corr}}}
\newcommand{\CR}{\mathop{\mathrm{CR}}}
\newcommand{\CM}{\mathop{\mathrm{CM}}}
\newcommand{\End}{\mathop{\mathrm{End}}}
\newcommand{\Hom}{\mathop{\mathrm{Hom}}}
\newcommand{\Spec}{\mathop{\mathrm{Spec}}}
\newcommand{\lcm}{\mathop{\mathrm{lcm}}}

\newcommand{\fonction}[5]{ #1~:\begin{array}{ccc}
 #2 & \longrightarrow & #3 \\
 #4 & \longmapsto & #5 \end{array}}

\newcommand{\ratrat}{\raisebox{2pt}{~~$\dashrightarrow$}\raisebox{-2pt}{$\!\!\!\!\!\!\!\!\!\!\dashleftarrow$~~}}

\newcommand{\implyss}[2]{$\mbox{(#1)}\Rightarrow\mbox{(#2)}$}

\newtheorem{thm}{Theorem}[section]
\newtheorem{theo}[thm]{Theorem}
\newtheorem{notation}[thm]{Notation}
\newtheorem{df}[thm]{Definition}
\newtheorem{prop}[thm]{Proposition}
\newtheorem{fact}[thm]{Fact}
\newtheorem{lem}[thm]{Lemma}
\newtheorem{cor}[thm]{Corollary}
\newtheorem{ex}[thm]{Example}

\renewcommand{\contentsname}{Contents\\{\footnotesize\normalfont(A table
of contents should normally not be included)}}
\begin{document}

\title{Upper motives of products of projective linear groups}
\author{Charles De Clercq}
\email{~declercq@math.jussieu.fr}
\address{Universit\'{e} Paris VI, 4 place Jussieu, 75252 Paris CEDEX 5}
%
%

\begin{abstract}
Fix a base field $F$, a finite field $\mathbb{F}$ and consider a sequence of central simple $F$-algebras $A_1,...,A_n$. In this note we provide some results toward a classification of the indecomposable motives lying in the motivic decompositions of projective homogeneous varieties under the action of $\PGL_1(A_1)$$\times$$ ...$$\times$$ \PGL_1(A_n)$ with coefficients in $\mathbb{F}$. We give a complete classification of those motives if $n=1$ and derive from it the motivic dichotomy of $\PGL_1$. We then provide several classification results as well as counterexamples for arbitrary $n$ showing that the situation is less rigid. These results involve a neat study of rational maps between generalized Severi-Brauer varieties which is certainly of independent interest.
\end{abstract}

\maketitle

\vspace*{6pt}


In the present paper we prove certain results toward a classification of the indecomposable motives which appear in the decompositions of the motives (with finite coefficients) of products of varieties of flags of (right) ideals in central simple algebras.

We provide with theorem \ref{classmotsb} a complete classification of the indecomposable motives (with coefficients in $\mathbb{F}_p$) lying in the decomposition of motives of projective $\PGL_1(A)$-homogeneous varieties, in terms of the subgroup generated by the class of the $p$-primary component of $A$ in the Brauer group of base field and the dimension of the underlying ideals. This result was formerly settled only for classical Severi-Brauer varieties in \cite{amitsur}. This classification enables to observe in theorem \ref{dichoto} a startling rigidity of these motives with respect of $A$, which we call the \emph{motivic dichotomy} of projective linear groups. 

We then study the case of indecomposable motives lying in the motivic decompositions of projective homogeneous varieties under the action of products of projective linear groups. We obtain under some assumptions some classification results in terms of the Brauer group of the base field. However we prove that these assumptions are mandatory by providing several counterexamples, and show that there is no analogue of the motivic dichotomy for projective homogeneous varieties under the action of products of projective linear groups.

Our approach relies on two main ingredients : the theory of upper motives of \cite{upper} and some results on the rational geometry of products generalized Severi-Brauer varieties. After introducing the background and notations, we provide in the third section the classification results of products of generalized Severi-Brauer varieties up to rational maps in both directions. These results are certainly of independent interest, and are then related with the theory of upper motives in the last section to get our motivic classification results.

\section{Context}

We fix a base field $F$, and by a \emph{variety} (over $F$) we will mean a smooth, projective scheme over $F$. The category of varieties over $F$ will be denoted by $\Var/F$, while the category of (commutative) algebras over $F$ will be denoted by $\Alg/F$.

\textbf{Central simple algebras and associated varieties.} An algebra $A$ over $F$ is \emph{central simple} if the center of $A$ is $F$ and if $A$ does not have any non-trivial two sided ideal. The $F$-dimension of a central simple algebra $A$ is a square and the \emph{degree} of $A$ (denoted by $\deg(A)$) is the square root of $\dim_F(A)$. Two central simple algebras $A$ and $B$ are \emph{Brauer equivalent} if $M_n(A)$ and $M_m(B)$ are isomorphic for some integers $n$ and $m$. The tensor product of algebras endows the set $\Br(F)$ of equivalence classes of central simple algebras under the Brauer equivalence with a structure of a torsion abelian group. The \emph{index} of a central simple algebra $A$ is the degree of the (uniquely determined up to isomorphism) division algebra Brauer equivalent to $A$. The \emph{exponent} of $A$ is the order of the class of $A$ in $\Br(F)$, which always divides $\ind(A)$. The $F$-dimension of any right ideal $I$ in $A$ is divisible by the degree of $A$, and the \emph{reduced dimension of $I$} is the integer $\rdim(I):=\dim_F(I)/\deg(A)$.

We now can define the classical varieties associated with a central simple algebra $A$ over $F$. The functor
$$\fonction{\PGL{}_1(A)}{\Alg/F}{Groups}{R}{\{\mbox{$R$-algebra automorphisms of $A\otimes_F R$}\}}$$
is representable by an affine algebraic group (see \cite[20.4]{KMRT}). One can associate to any sequence $0\leq d_1 <...<d_k\leq \deg(A)$ the \emph{variety of flags of (right) ideals} of reduced dimension $d_1,...,d_k$ in $A$, defined by the representable functor
$$\fonction{X(d_1,...,d_k;A)}{\Alg/F}{Set}{R}{\left\{
\begin{array}{c}
\mbox{{\footnotesize sequences $I_1\subset ...\subset I_k$ of right ideals of $A\otimes R$ of reduced}}\\
\mbox{{\footnotesize dimension $d_1,...,d_k$ such that the injections $I_j\rightarrow A\otimes R$ split.}}\end{array}\right\}}.$$

Important examples of varieties of flags of right ideals in a central simple algebra $A$ are certainly the \emph{classical Severi-Brauer variety} $X(1;A)$ and the \emph{generalized Severi-Brauer varieties} $X(k;A)$.

Let $A_1,...,A_n$ be a sequence of central simple algebras over $F$ and consider the algebraic group $G=\PGL_1(A_1)\times ...\times \PGL_1(A_n)$. Any $G$-homogeneous variety is isomorphic to a product $X_1\times...\times X_n$, where for any $1\leq j \leq n$, $X_j$ is a variety of flags of right ideal in $A_j$.

\textbf{Grothendieck Chow motives.} The category Grothendieck Chow motives was introduced by Grothendieck as a linearization of the category of varieties over $F$, replacing a morphism of varieties $f:X\rightarrow Y$ by the algebraic cycle (modulo rational equivalence) defined by its graph on the product $X\times Y$. For any $X\in \Var/F$ and any commutative ring $\Lambda$, the group of \emph{$i$-dimensional cycles on $X$ modulo rational equivalence} with coefficients in $\Lambda$ is defined by $\CH_i(X;\Lambda)=\CH_i(X)\otimes \Lambda$ (we refer to \cite{fulton} or \cite{EKM} for the definition of Chow groups).

\begin{notation}
If $Y\in \Var /F$ is a variety and $X=\bigsqcup_{i=1}^nX_i$ the decomposition of another variety $X$ into irreducible components, the group of \emph{correspondences of degree $k$} between $X$ and $Y$ with coefficients in $\Lambda$ is defined by $\Corr_k(X,Y;\Lambda)=\bigoplus_{i=1}^n\CH_{\dim(X_i)+k}(X_i\times Y;\Lambda)$. 
\end{notation}
\newpage
For any three varieties $X$, $Y$ and $Z$ the three projections

$$\xymatrix{
   & X\times Y\times Z \ar_{\pi_2}[ld]\ar^{\pi_1}[d]\ar^{\pi_3}[rd]& \\
    X\times Y & X\times Z & Y\times Z
  }$$

\vspace{0,2cm}

\noindent give rise to a bilinear pairing $\cdot \circ \cdot:\Corr_k(Y,Z;\Lambda)\otimes \Corr_{k'}(X,Y;\Lambda)\rightarrow \Corr_{k+k'}(X,Z;\Lambda)$ given by $\beta\circ \alpha:=\left.\pi_1\right._{\ast}\left(\pi_2^{\ast}(\alpha)\cdot \pi_3^{\ast}(\beta)\right)$ (see \cite[\S 63]{EKM}).

First, we denote by $\CR(F;\Lambda)$ the \emph{category of correspondences} with coefficients in $\Lambda$, defined as follows. Its objects are finite direct sums of pairs $X[i]$, where $X\in \Var/k$ is a variety and $i$ an integer. A morphism between two objects $X[i]$ and $Y[j]$ is a correspondence of $\Corr_{i-j}(X,Y;\Lambda)$, the composition being the one previously defined.

The category $\CM(F;\Lambda)$ is then defined as the pseudoabelian envelope of $\CR(F;\Lambda)$. Its objects are couples $(X,\pi)$, where $X$ is an object of $\CR(F;\Lambda)$ and $\pi$ a projector in $\End_{\CR(F;\Lambda)}(X)$. Abusing notation, we will denote by $(X,\pi)[i]$ the direct summand of a twist of a variety $X[i]$ defined by a projector $\pi \in \Corr_0(X,X;\Lambda)$. The morphisms of $\CM(F;\Lambda)$ are given by the formula 
$$\Hom{}_{\!\CM(F,\Lambda)}\left((X,\pi),(Y,\rho)\right)=\rho\circ \Hom{}_{\!\CR(F,\Lambda)}\left(X,Y\right)\circ \pi.$$

Finally for any commutative ring $\Lambda$, one gets the realization functor
$$\begin{array}{ccc}
\Var/F&\longrightarrow &\CM(F;\Lambda)\\
X&\longmapsto&M(X):=(X,\Gamma_{id_X})[0]\\
f:X\rightarrow Y&\longmapsto &\Gamma_f\otimes 1\end{array}$$
where $\Gamma_f$ denotes the cycle defined by the graph of $f$.

The set of the \emph{Tate motives} in $\CM(F;\Lambda)$ is the set of all the twists of $M(\Spec(F))$, which will be denoted by $\{\Lambda[i],~i\in \mathbb{Z}\}$. For any motive $M\in \CM(F;\Lambda)$, the \emph{i-th Chow group} of $M$ is defined by $\CH_i(M;\Lambda):=\Hom(\Lambda[i],M)$.

\section{Upper motives}

$\indent$From now on we assume that the ring of coefficients $\Lambda$ is finite and connected and we fix an algebraic closure $\overline{F}$ of $F$. Let $G$ be a direct product $\PGL_1(A_1)\times...\times \PGL_1(A_n)$ of projective linear groups associated to some central simple algebras $A_1,...,A_n$ over $F$ and let $X$ be a $G$-homogeneous variety.

By \cite{kock} (see also \cite{chermergil}) the motive $M(X_{\overline{F}})\in \CM(F;\Lambda)$ is isomorphic to a finite direct sum of Tate motives. Moreover by \cite{chermer} $X$ satisfies Rost nilpotence principle and thus 
the results of \cite{chermer} (see also \cite{upper}) show that the motive of $X$ decomposes in an essentially unique way as a direct sum of indecomposable motives.

The theory of upper motives was introduced in \cite{upper} to describe the indecomposable objects of $\CM(F;\Lambda)$ lying in the decomposition of homogeneous varieties under the action of a semisimple affine algebraic group. We will only focus in this note on this description in the case where the algebraic group $G$ is a product of projective linear groups. We will also stick to motives of coefficients in $\mathbb{F}_p$, as the results of \cite{decvarprojcoeff} show these are the essential case.\\

\newpage

\begin{df}Let $X$ be a $\PGL_1(A_1)$$\times$$...$$\times$$\PGL_1(A_n)$-homogeneous variety. The \emph{upper $p$-motive} of $X$ is defined as the isomorphism class of the indecomposable summand of $M(X)\in\CM(F;\mathbb{F}_p)$ whose $0$-codimensional Chow group is non-zero.
\end{df}

In the sequel, for any sequence of $p$-primary division algebras $D_1,...,D_n$, the upper $p$-motive of the variety $X(p^{k_1};D_1)\times ...\times X(p^{k_n};D_n)$ will be denoted by $M_{D_1,...,D_n}^{k_1,...,k_n}$.

\begin{prop}\label{upp}Let $G=\PGL_1(A_1)\times...\times \PGL_1(A_n)$ be a product of projective linear groups associated some central simple algebras $A_1,...,A_n$. Any indecomposable summand of the motive $M(X)\in \CM(F;\mathbb{F}_p)$ of a $G$-homogeneous variety $X$ is isomorphic to a shift of $M^{k_1,...,k_n}_{D_1,...,D_n}$, where for any $1\leq j \leq n$, $D_j$ is the division algebra Brauer-equivalent to the $p$-primary component of $A_j$.
\end{prop}
\begin{proof}One may either directly apply \cite[Theorem 3.5]{upper} to the case where $G$ is a product of projective linear groups or apply \cite[Theorem 3.8]{upper} to each factor.
\end{proof}

Proposition \ref{upp} implies that the classification of the indecomposable summands of the motives of $\PGL_1(A_1)\times...\times \PGL_1(A_n)$-homogeneous varieties with coefficients in $\mathbb{F}_p$ is reduced to the classification of the upper $p$-motives $M_{D_1,...,D_n}^{k_1,...,k_n}$, where for any $1\leq i \leq n$, $D_i$ is a $p$-primary division algebra and $0\leq k_i<v_p(\deg(D_i))$. The next section is dedicated to the study of rational maps between such products of generalized Severi-Brauer varieties, and we will study the connections between this classical problem and the classification of those upper $p$-motives in the last section.

\section{Rational maps and product of generalized Severi-Brauer varieties}

$\indent$In this section we introduce the needed tools to the classification results of the indecomposable motives of homogeneous varieties under the action of a product of projective linear groups. These tools involve a refined study of the rational maps between products of generalized Severi-Brauer varieties.

\begin{notation}Let $k_1,...,k_n$ be a sequence of integers and $D,D_1,...,D_n$ be $p$-primary division algebras of degree $p^s$ for some prime $p$. Let us define the function
$$\fonction{\mu^{k_1,...,k_n}_{D,D_1,...,D_n}}{\mathbb{N}^n}{\mathbb{N}}{(i_1,...,i_n)}{\frac{p^{k_1}}{\gcd(i_1,p^{k_1})}\cdots\frac{p^{k_n}}{\gcd(i_n,p^{k_n})}\ind(D\otimes D_1^{\otimes -i_1}\otimes...\otimes D_n^{\otimes -i_n})}.$$
\end{notation}

In the sequel the following index reduction formula due to Merkurjev, Panin and Wadsworth (see \cite{flag}) will be of constant use.

\begin{fact}\label{indexx}Let $p$ be a prime and $D,D_1,...,D_n$ be division algebras of degree $p^s$ over $F$. For any sequence of integers $0\leq k_1,...,k_n\leq s$ the index of $D$ over the function field of the variety $X(p^{k_1};D_1)\times ...\times X(p^{k_n};D_n)$ is given by
$$\ind(D_{F(X(p^{k_1};D_1)\times ...\times X(p^{k_n};D_n))})=\min_{1\leq i_1,...,i_n\leq p^s}\mu^{k_1,...,k_n}_{D,D_1,...,D_n}(i_1,...,i_n).$$
\end{fact}

\textbf{Products of classical Severi-Brauer varieties}. Let $p$ be a prime and $D_1,...,D_n,D'_1,...,D'_m$ be a sequence of division algebras over $F$ of degree $p^s$. We give here an algebraic criterion for the existence of rational maps $X(1;D_1)\times ...\times X(1;D_n) \ratrat X(1;D'_1)\times ...\times X(1;D'_m)$ in terms of the subgroup of the Brauer group of $F$ generated by the classes of the underlying algebras.

\begin{prop}\label{class}Let $p$ be a prime and $D_1,...,D_n,D'_1,...,D'_m$ be a sequence of division algebras of degree $p^s$ over $F$, the following assertions are equivalent :
\begin{enumerate}
\item there are two rational maps 
$$X(1;D_1)\times...\times X(1;D_n)\ratrat X(1;D'_1)\times...\times X(1;D'_m);$$
\item the classes of $D_1,...,D_n$ and the classes of $D'_1,...,D'_m$ generate the same subgroup in the Brauer group of $F$.
\end{enumerate}
\end{prop}

\begin{proof}
\implyss{i}{ii} By symmetry, it suffices to show that for any $1\leq j\leq n$, the existence of a rational map $X(1;D'_1)\times...\times X(1;D'_m)\dashrightarrow X(1;D_j)$ implies that the class of $D_j$ lies in the subgroup generated by the classes of $D'_1,...,D'_m$ in the Brauer group of $F$.

Let us denote by $E$ the function field $F(X(1;D'_1)\times...\times X(1;D'_m))$. The existence of a rational point on $X(1;D_j)$ over $E$ implies that $D_j$ splits over $E$. By fact \ref{indexx} this implies that there are some integers $i_{j,1},...,i_{j,m}$ such that
$$\mu^{0,...,0}_{D_j,D'_1,...,D'_m}(i_{j,1},...,i_{j,m})=\ind(D_j\otimes D'_1{}^{\otimes-i_{j,1}}\otimes ...\otimes D'_m{}^{\otimes -i_{j,m}})=1$$
hence $[D_j]=[D'_1]^{i_{j,1}}...[D'_m]^{i_{j,m}}$ in the Brauer group of $F$.

\vspace{0,5cm}

\implyss{ii}{i} As for the first implication, it suffices to show that if for some $0\leq j \leq n$ the class of $D_j$ lies in the subgroup of the Brauer group of $F$ generated by the classes of $D'_1,...,D'_m$, there is a rational map $X(1;D'_1)\times...\times X(1;D'_m)\dashrightarrow X(1;D_j)$. Assume that $[D_j]=[D'_1]^{i_{j,1}}...[D'_m]^{i_{j,m}}$ for some integers $i_{j,1},...,i_{j,m}$. Since $\mu_{D_j,D'_1,...D'_m}^{0,...,0}(i_{j,1},...,i_{j,m})=1$, the division algebra $D_j$ splits over the function field $E=F(X(1;D'_1)\times...\times X(1;D'_m))$. The variety $X(1;D_j)$ has therefore a rational point over $E$ and we get the needed rational map.
\end{proof}

\textbf{Products of generalized Severi-Brauer varieties}. We now consider the case of products of generalized Severi-Brauer varieties. As for the classical ones, we would like to give an algebraic criterion for the existence of rational maps in both directions between products of Severi-Brauer varieties in terms of the subgroups generated by the classes of the underlying algebras in the Brauer group of the base field. We will give such a criterion with some assumptions on the exponents of the algebras, and show that one can't expect to release this restriction by producing some counterexamples.

\begin{lem}\label{lemma}Let $p$ be a prime and $D,D_1,...,D_n$ division algebras of degree $p^s$. Assume that
\begin{enumerate}
\item $\exp(D)\geq \max_{i=1,...,n}\exp(D_i)$; 
\item for some $0\leq k< s$, there is a rational map
$$X(p^k,D_1)\times ...\times X(p^k,D_n)\dashrightarrow X(p^k,D).$$
\end{enumerate}
Then there are some integers $i_1,...,i_n$ which satisfy $\sum_{j=1}^nv_p(\gcd(i_j,p^k))=k(n-1)$ and such that $[D]=[D_1]^{i_1}...[D_n]^{i_n}$ in the Brauer group of $F$.
\end{lem}

\begin{proof}We denote by $E$ the function field $F(X(p^k,D_1)\times ...\times X(p^k,D_n))$. By assumption $\mbox{(ii)}$, $X(p^k,D)$ has a rational point over $E$, hence $\ind(D_E)$ divides $p^k$. There is therefore by fact \ref{indexx} a sequence of integers $(i_1,...,i_n)$ such that $\mu_{D,D_1,...,D_n}^{k,...,k}(i_1,...,i_n)$ divides $p^k$. Computing the $p$-adic valuation of $\mu_{D,D_1,..,D_n}^{k,...,k}(i_1,...,i_n)$, we get the inequality
$$nk-\sum_{j=1}^nv_p(\gcd(i_j,p^k))\leq k. (\ast)$$
We claim that inequality $(\ast)$ is in fact an equality. Assume that $(\ast)$ is not an equality, that is to say that $(n-1)k<\sum_{j=1}^nv_p(\gcd(i_j,p^k))$. Since for any $1\leq j\leq n$, $v_p(\gcd(i_j,p^k))$ is at most equal to $k$, this inequality implies that for any such $j$, $v_p(\gcd(i_j,p^k))$ is non-zero. As a direct consequence and for any $j$, the exponent of the algebra $D_j^{i_j}$ is strictly lesser than $\exp(D)$ by assumption $\mbox{(i)}$, and in particular
$$\exp(D\otimes D_1^{\otimes -i_1}\otimes ...\otimes D_n^{\otimes -i_n})=\exp(D).$$

Since $\mu_{D,D_1,...,D_n}^{k,...,k}(i_1,...,i_n)$ divides $p^k$ and the exponent of a central simple algebras always divides its index, the $p$-adic valuation of $\exp(D)$ is at most $\left(\sum_{j=1}^nv_p(\gcd(i_j,p^k))\right)-(n-1)k$. Assumption $\mbox{(i)}$ then implies that for any $1\leq j \leq n$, $\exp(D_j)\mid \exp(D)\mid p^{v_p(\gcd(i_j,p^k))}$, hence $D_j^{i_j}$ is split. Going back to the expression of $\mu_{D,D_1,...,D_n}^{k,...,k}(i_1,...,i_n)$, we have
$$\ind(D)=\ind(D\otimes D_1^{\otimes -i_1}\otimes ...\otimes D_n^{\otimes -i_n})\mid p^{\sum_{j=1}^nv_p(\gcd(i_j,p^k))-(n-1)k}$$
and thus $\ind(D)$ divides $p^k$, a contradiction.

We have therefore shown the claim, i.e. $\sum_{j=1}^nv_p(\gcd(i_j,p^k))=(n-1)k$. It is then clear that since
$$\ind(D_E)=\mu_{D,D_1,...,D_n}^{k,...,k}(i_1,...,i_n)=p^k\ind(D\otimes D_1^{\otimes -i_1}\otimes ...\otimes D_n^{\otimes -i_n})$$
divides $p^k$, the equality $[D]=[D_1]^{i_1}...[D_n]^{i_n}$ holds in $\Br(F)$. 
\end{proof}

\vspace{0,2cm}

We now look at the consequences of lemma \ref{lemma} in the classification of products of generalized Severi-Brauer varieties up to rational maps in both directions. The motivic consequences of these results will be discussed in the next section.

\vspace{0,2cm}

\textbf{Classification for generalized Severi-Brauer varieties}. The following theorem shows that up to rational maps in both directions, non-trivial generalized Severi-Brauer varieties of right ideals of $p$-primary reduced dimension in $p$-primary division algebras are completely classified by the subgroup of the Brauer group generated by the classes underlying algebras and the dimension of the underlying ideals.

\begin{theo}\label{classratsb}
Let $p$ be a prime, and $D$, $D'$ be two $p$-primary division algebras over $F$. Assume that $0\leq k<\deg(D)$ and $0\leq k'<\deg(D')$. There are two rational maps $X(p^k;D)\ratrat X(p^{k'};D')$ if and only if $k=k'$ and the classes of $D$ and $D'$ generate the same subgroup of $\Br(F)$.
\end{theo}

Theorem \ref{classratsb} generalizes \cite[Theorem 9.3]{amitsur}, which corresponds to the case where the two integers $k$ and $k'$ are equal to $0$. To prove theorem \ref{classratsb} we will need the following result.

\begin{theo}\label{classn1}Let $p$ be a prime and $D$, $D'$ be two division algebras of degree  $p^s$ over $F$. The following assertions are equivalent :
\begin{enumerate}
\item for some $0\leq k<s$, there are two rational maps $X(p^k;D)\ratrat X(p^k;D')$;
\item the classes of $D$ and $D'$ generate the same subgroup in the Brauer group of $F$;
\item for any $0\leq k<s$, there are two rational maps $X(p^k;D)\ratrat X(p^k;D')$.
\end{enumerate}
\end{theo}

\begin{proof}
We show the implications (i) $\Rightarrow$ (ii) $\Rightarrow$ (iii) and we begin with \implyss{i}{ii}. The case where $k=0$ corresponds to the \implyss{i}{ii} of proposition \ref{class}, with $n=m=1$. To show that the implication also hold if $k$ is not $0$, we may exchange $D$ and $D'$ and thus assume that the exponent of $D$ is greater than the exponent of $D'$. Lemma \ref{lemma} applied with $n=1$ gives us an integer $i_1$ coprime to $p$ such that in the Brauer group of $F$, $[D]=[D']^{i_1}$. It is then clear that the classes of $D$ and $D'$ generate the same subgroup of the Brauer group of $F$.

Condition $\mbox{(ii)}$ implies $\mbox{(iii)}$. Indeed if the classes of $D$ and $D'$ generate the same subgroup of $\Br(F)$, then for any field extension $E/F$ the equality $\ind(D_E)=\ind(D'_E)$ holds. In particular for any $0\leq k < n$ the variety $X(p^k;D)$ has a rational point over the function field of $X(p^k;D')$ and vice versa. There are therefore two rational maps $X(p^k;D)\ratrat X(p^k;D')$.
\end{proof}

\begin{proof}[Proof of theorem \ref{classratsb}]If the subgroups generated by the classes of $D$ and $D'$ coincide, then for any $0\leq k<n$ there are two rational maps $X(p^k;D)\ratrat X(p^k;D')$ by implication \implyss{ii}{iii} of theorem \ref{classn1}.

For the other implication, we first observe that if there are two rational maps in both directions $X(p^k;D)\ratrat X(p^{k'};D')$, then $k=k'$ and $\deg(D)=\deg(D')$. Indeed if we set $\deg(D)=p^n$ and $\deg(D')=p^{n'}$, the existence of two rational maps $X(p^k;D)\ratrat X(p^{k'};D')$ implies that the canonical dimension of those varieties are equal. By \cite{upper}, those varieties are incompressible and thus the dimension of $X(p^k;D)$ (which is $p^k(p^n-p^k)$) is equal to the dimension of $X(p^{k'};D')$. The equality $p^k(p^n-p^k)=p^{k'}(p^{n'}-p^{k'})$ then implies that $k=k'$, and $n=n'$. It remains to use assertion \implyss{i}{ii} of theorem \ref{classn1} to observe that the classes of $D$ and $D'$ generate the same subgroup of the Brauer group of $F$.
\end{proof}
\vspace{0,2cm}

\textbf{Products of generalized Severi-Brauer varieties}. As for the case of generalized Severi-Brauer varieties, we would like to completely determine products of generalized Severi-Brauer varieties of $p$-primary ideals in $p$-primary division algebras up to rational maps in both directions. We will see that under some assumptions on the exponents of the algebras, there can still recover some results in terms of the Brauer classes of the underlying algebras. We will besides produce counterexamples showing that the classification does not rely on the subgroups of $\Br(F)$ outside this setting.

\begin{theo}\label{prodexp}Let $p$ be a prime and $D_1,...,D_n,D'_1,...,D'_m$ be division algebras of degree $p^s$. Consider an integer $0\leq k <s$ and assume that $\exp(D_1)$$=$$...$$=$$\exp(D_n)$ and $\exp(D'_1)$$=$$...$$=$$\exp(D'_m)$. The following assertions are equivalent :
\begin{enumerate}
\item There are two rational maps 
$$X(p^k;D_1)\times ...\times X(p^k;D_n)\ratrat X(p^k;D'_1)\times...\times X(p^k;D'_m);$$
\item for any $1\leq i \leq n$ and any $1\leq j \leq m$, there are integers $\alpha_{i,1},...,\alpha_{i,m}$, and $\beta_{j,1},...,\beta_{j,n}$ such that in $\Br(F)$
$$[D_i]=[D'_1]^{\alpha_{i,1}}...[D'_m]^{\alpha_{i,m}}\mbox{ and }[D'_j]=[D_1]^{\beta_{j,1}}...[D_n]^{\beta_{j,n}}$$
with $\sum_{u=1}^mv_p(\gcd(\alpha_{i,u},p^k))=(m-1)k$ and $\sum_{v=1}^nv_p(\gcd(\alpha_{j,v},p^k))=(n-1)k$.
\end{enumerate}
\end{theo}

\begin{proof}\implyss{i}{ii} We may exchange the division algebras $D_i$'s by the $D'_j$'s and thus assume that $\exp(D_1)\geq \exp(D'_1)$. By assumption (i), for any integer $1\leq i \leq n$ there is a rational map $X(p^k;D'_1)\times ...\times X(p^k;D'_m)\dashrightarrow X(p^k;D_i)$. Lemma \ref{lemma} then gives us for any such $i$ a sequence $\alpha_{i,1},...,\alpha_{i,m}$ satisfying $\sum_{u=1}^mv_p\gcd(\alpha_{i,u},p^k)=(m-1)k$ and such that the equality $[D_i]=[D'_1]^{\alpha_{i,1}}...[D'_m]^{\alpha_{i,m}}$ holds in the Brauer group of $F$.

It follows that $\exp(D_1)=\exp(D'_1{}^{\alpha_{1,1}}\otimes ...\otimes D'_m{^{\alpha_{1,m}}})\leq \lcm_{i=1,...,m}\exp(D'_i)$, and consequently $\exp(D_1)=\exp(D'_1)$. The rational maps $X(p^k;D_1)\times ...\times X(p^k;D_n)\dashrightarrow X(p^k;D'_j)$ for any $1\leq j\leq m$ together with lemma \ref{lemma} thus give the needed sequence $(\beta_{j,v})_{v=1,...,n}$.

\vspace{0,2cm}

\implyss{ii}{i} We denote by $E$ (resp. by $E'$) the function field of $X(p^k;D_1)\times...\times X(p^k;D_n)$ (resp. of the variety $X(p^k;D'_1)\times...\times (p^k;D'_m)$). For any $0\leq i \leq n$, $\mu_{D_i,D'_1,...,D'_m}^{k,...,k}(\alpha_{i,1},...,\alpha_{i,m})$ is equal to $p^k$. The variety $X(p^k;D_i)_{E'}$ therefore has a rational point by fact \ref{indexx} and there is a rational map $X(p^k;D'_1)\times ...\times (p^k;D'_m)\dashrightarrow X(p^k;D_1)\times... \times (p^k;D_n)$. The same procedure with the sequences $(\beta_{j,v})_{j,v}$ shows that for any $0\leq j \leq m$ the variety $X(p^k;D_j)_E$ has a rational point hence we get the rational map in the other direction.
\end{proof}

\begin{cor}\label{corprodexp}Let $p$ be a prime and $D_1,...,D_n,D'_1,...,D'_m$ be division algebras of degree $p^s$. Assume that $0\leq k <s$, $\exp(D_1)$$=$$...$$=$$\exp(D_n)$ and also $\exp(D'_1)$$=$$...$$=$$\exp(D'_m)$. If there are two rational maps $X(p^k;D_1)\times ...\times X(p^k;D_n)\ratrat X(p^k;D'_1)\times...\times X(p^k;D'_m)$, then there are also rational maps $X(1;D_1)\times ...\times X(1;D_n)\ratrat X(1;D'_1)\times...\times X(1;D'_m)$.
\end{cor}

\begin{proof}Assertion (ii) of theorem \ref{prodexp} implies assertion (ii) of proposition \ref{class}.
\end{proof}

We would like to point out that the assertion (ii) of theorem \ref{prodexp} is strictly more restrictive than the fact that the subgroups of $\Br(F)$ generated by the classes of $D_1$,...,$D_n$ and $D'_1$,...,$D'_m$ coincide. This question is addressed in the next section. 

\vspace{0,2cm}

\textbf{Some counterexamples} Theorem \ref{classn1} shows that for $p$-primary division algebras of same degree $D$ and $D'$, one may check that there are two rational maps between some generalized Severi-Brauer varieties $X(p^k;D)\ratrat X(p^k;D')$ (with $0\leq k <v_p(\deg(D))$) by looking at \emph{classical} Severi-Brauer varieties, whose splitting behavior is much simpler. We now show that the situation is more intricate for products of generalized Severi-Brauer varieties. These counterexamples show that the assumptions of theorem \ref{prodexp} and \ref{corprodexp} are in fact minimal and that there is no hope to find an analogue of theorem \ref{classn1} for products of generalized Severi-Brauer varieties.

\begin{ex}\label{ex1}Consider three biquaternion division algebras $\Delta_1:=D_1\otimes D_2$, $\Delta_2:=D_1\otimes D_3$ and $\Delta_3:=D_2\otimes D_3$ over $F$. The subgroup of $\Br(F)$ generated by the classes $[\Delta_1]$ and $[\Delta_2]$ clearly coincides with the subgroup generated by $[\Delta_1]$ and $[\Delta_3]$. We then know by proposition \ref{class} that there are two rational maps $X(1;\Delta_1)\times X(1;\Delta_2)\ratrat X(1;\Delta_1)\times X(1;\Delta_3)$, but one can check that there are no rational maps between the associated products of generalized Severi-Brauer varieties since there is no rational map $X(2;\Delta_1)\times X(2;\Delta_2)\dashrightarrow X(2;\Delta_3)$.
\end{ex}

\begin{ex}\label{ex2}Setting $F=\mathbb{Q}(x_1,x_2,x_3)$, we now show that there is no criterion for the existence of rational maps in both directions between products of generalized Severi Brauer varieties in $p$-primary division algebras in terms of the subgroups of $\Br(F)$. Consider the field extension $\mathbb{Q}(\zeta_5)/\mathbb{Q}$ generated by a primitive $5$-th root of the unity $\zeta_5$. The cyclic division algebra $D_1=(\mathbb{Q}(\zeta_5)(x_1,x_2,x_3),x_1)$ over $F$ is of index $4$ and exponent $4$. We may also consider the two biquaternion algebras $D_2=(\mathbb{Q}(\sqrt{5})(x_1,x_2,x_3),x_1)\otimes (\mathbb{Q}(\sqrt{2})(x_1,x_2,x_3),x_2)$ and $D_3=(\mathbb{Q}(\sqrt{5})(x_1,x_2,x_3),x_1)\otimes (\mathbb{Q}(\sqrt{3})(x_1,x_2,x_3),x_3)$. The central simple algebra $D_2$ (resp. $D_3$) is a division algebra of index $4$ and exponent $2$ since the extensions $\mathbb{Q}(\sqrt{5})$ and $\mathbb{Q}(\sqrt{2})$ (resp. $\mathbb{Q}(\sqrt{3})$) are linearly disjoint. The subgroups generated by the classes of $D_1$, $D_2$ and the classes of $D_1$, $D_3$ in the Brauer group of $F$ do not coincide, but there are two rational maps $X(2;D_1)\times X(2;D_2)\ratrat X(2;D_1)\times X(2;D_3)$. The assumption on the exponents of the algebras is therefore mandatory in both theorem \ref{prodexp} and corollary \ref{corprodexp}.
\end{ex}

\section{Consequences on motives}

In this section we will relate the previous results on the classification of products of generalized Severi-Brauer varieties up to rational maps in both directions to the classification of the upper $p$-motives of $\PGL_1(A_1)$$\times$$...$$\times$$\PGL_1(A_n)$-homogeneous varieties.

\begin{notation}
Let $A_1,...,A_n$ be a sequence of central simple algebras over the field $F$. We denote by $\mathfrak{X}^{p}_{A_1,...,A_n}$ the set of all the isomorphism classes of twists of indecomposable direct summands of projective $\PGL_1(A_1)\times ...\times \PGL_1(A_n)$-homogeneous varieties in $\CM(F;\mathbb{F}_p)$.
\end{notation}

\begin{lem}\label{lemma2}Let $p$ be a prime and $A_1,...,A_n$ be central simple algebras of $p$-primary index. Assume that $X=X(k_1;A_1)\times...\times X(k_n;A_n)$ is a product of generalized Severi-Brauer varieties. The variety $X$ has a rational point if and only if $X$ has a $0$-cycle of degree coprime to $p$.
\end{lem}
\begin{proof}It is clear that the existence of a rational point on $X$ gives a $0$-cycle of degree coprime to $p$. Conversely if $X=X(k_1;A_1)\times...\times X(k_n;A_n)$ has a $0$-cycle of degree coprime to $p$, then there is a field extension $E/F$ of degree coprime to $p$ such that $X_E$ has a rational point. We thus get a rational point on each $X(k_j,A_j)_E$ but since the central simple algebras $A_1,...,A_n$ are $p$-primary, their index remains the same over $E$. Therefore each variety $X(k_j,A_j)$ has a rational point and $X$ has a rational point.
\end{proof}

\begin{prop}\label{link}
Let $p$ be a prime and $D_1,...,D_n$, $D'_1,...,D'_n$ be $p$-primary division algebras. The following assertions are equivalent :
\begin{enumerate}
\item the upper $p$-motives $M_{D_1,...,D_n}^{k_1,...,k_n}$ and $M_{D'_1,...,D'_m}^{k'_1,...,k'_m}$ are isomorphic;
\item there are two rational maps
$$X(p^{k_1};D_1)\times...\times X(p^{k_n};D_n)\ratrat X(p^{k'_1};D'_1)\times...\times X(p^{k'_m};D'_m).$$
\end{enumerate}
\end{prop}
\begin{proof}
By \cite[Corollary 2.15]{upper}, the upper $p$-motives of $X=X(p^{k_1};D_1)\times...\times X(p^{k_n};D_n)$ and $Y=X(p^{k'_1};D'_1)\times...\times X(p^{k'_m};D'_m)$ are isomorphic if and only if both $X_{F(Y)}$ and $Y_{F(X)}$ have a $0$-cycle of degree coprime to $p$. Lemma \ref{lemma2} states that this is equivalent to the fact that both $X_{F(Y)}$ and $Y_{F(X)}$ have a rational point.
\end{proof}

\textbf{Classification of $\mathfrak{X}^{p}_{A}$.} The previous discussion gives the following complete classification of the indecomposable $p$-motives of $\PGL_1(A)$-homogeneous varieties. This classification leads to a startling rigidity of the sets $\mathfrak{X}^p_A$ with respect to the parameter $A$, described in theorem \ref{dichoto}.

\begin{theo}\label{classmotsb}Let $p$ be a prime and $D, D'$ be $p$-primary division algebras over $F$. Consider two integers $0\leq k <v_p(\deg(D))$ and $0\leq k' <v_p(\deg(D'))$. The upper $p$-motives $M_{k,D}$ and $M_{k,D'}$ are isomorphic if and only if $k=k'$ and the subgroups generated by the classes of $D$ and $D'$ in $\Br(F)$ coincide.
\end{theo}
\begin{proof}Proposition \ref{link} asserts that the fact that $M_{k,D}$ and $M_{k',D}$ are isomorphic is equivalent to the existence of rational maps in both directions between the associated generalized Severi-Brauer varieties. Theorem \ref{classratsb} then gives the expected criterion.
\end{proof}

\begin{theo}[{Motivic dichotomy of $\PGL_1$}]\label{dichoto} Let $A$ and $A'$ be two central simple algebras over $F$. Then either $\mathfrak{X}^p_A\cap \mathfrak{X}^p_{A'}$ is reduced to the Tate motives or $\mathfrak{X}^p_A=\mathfrak{X}^p_{A'}$.
\end{theo}
\begin{proof}Denote by $D$ and $D'$ the division algebras Brauer equivalent to the $p$-primary components of $A$ and $A'$. The class of any motive $M\in\mathfrak{X}^p_A\cap \mathfrak{X}^p_{A'}$ is isomorphic to the same twist of a motive $M^k_D$ and a motive $M^{k'}_{D'}$ for some integers $k$ and $k'$ by proposition \ref{upp}. If this motive is non-Tate, theorem \ref{classratsb} implies that the subgroups of $\Br(F)$ generated by $[D]$ and $[D']$ coincide, hence any isomorphism class of an indecomposable motive $M^k_{D}[i]\in \mathfrak{X}^p_A$ is the isomorphism class of $M^k_{D'}[i]$, and thus lies in $\mathfrak{X}^p_{A'}$.
\end{proof}

\textbf{Classification of $\mathfrak{X}^{p}_{A_1,...,A_n}$.} The previous section gave a glimpse on the difficulties which appear when dealing with products of generalized Severi-Brauer varieties. We now investigate the results which can still be deduced from it and show that the motivic dichotomy does not hold for products of projective linear groups.

\begin{theo}
Let $p$ be a prime and $D_1,...,D_n,D'_1,...,D'_n$ be division algebras of degree $p^s$ over $F$. Assume that $\exp(D_1)$$=$$...$$=$$\exp(D_n)$, $\exp(D'_1)$$=$$...$$=$$\exp(D'_m)$ and $0\leq k < s$. The upper $p$-motives $M^{k,...,k}_{D_1,...,D_n}$ and $M^{k,...,k}_{D'_1,...,D'_m}$ are isomorphic if and only if for any $1\leq i \leq n$ and for any $1\leq j \leq m$, there are two sequences $(\alpha_{i,u})_{1\leq u \leq m}$ and $(\beta_{j,v})_{1\leq v \leq n}$ such that :
\begin{enumerate}
\item $[D_i]=[D'_1]^{\alpha_{i,1}}... [D'_m]^{\alpha_{i,m}}$ and $[D'_j]=[D_1]^{\beta_{j,1}}\cdot ...\cdot [D_n]^{\beta_{j,n}}$ in $\Br(F)$;
\item $\sum_{u=1}^mv_p(\gcd(\alpha_{i,u},p^k))=(m-1)k$ and $\sum_{v=1}^nv_p(\gcd(\beta_{j,v},p^k))=(n-1)k$.
\end{enumerate}
\end{theo}

\begin{proof}By proposition \ref{link} the classification of those upper motives corresponds exactly to the classification of theorem \ref{prodexp}.
\end{proof}

The results of the previous section however imply that theorem \ref{dichoto} does not hold for products of generalized Severi-Brauer varieties.

\begin{theo}
There is no analogue of the motivic dichotomy of $\PGL_1$ for products of projective linear groups.
\end{theo}
\begin{proof}Consider te three biquaternion division algebras $\Delta_1$, $\Delta_2$ and $\Delta_3$ of example \ref{ex1}. By proposition \ref{link} the indecomposable $2$-motives $M_{\Delta_1,\Delta_2}^{0,0}$ and $M_{\Delta_1,\Delta_3}^{0,0}$ are isomorphic, whereas $M_{\Delta_1,\Delta_2}^{1,1}$ is not isomorphic to an upper $2$-motive of $\PGL_1(\Delta_1)\times \PGL_1(\Delta_3)$. The set $\mathfrak{X}_{\Delta_1,\Delta_2}^2\cap \mathfrak{X}_{\Delta_1,\Delta_3}^2$ is then neither equal to $\mathfrak{X}_{\Delta_1,\Delta_2}^2$ nor reduced to the Tate motives.
\end{proof}

\begin{acknowledgements}
I would like to thank Nikita Karpenko for raising this question.
\end{acknowledgements}


\end{document}